\documentclass{amsart}
\usepackage[T1]{fontenc}
\newtheorem{theorem}[subsection]{Theorem}

\title{Milnor number of an invariant singularity: generalization of Chulkov's inequality}
\author{Ivan Proskurnin}
\date{}

\begin{document}
\maketitle

\begin{abstract}

We prove a lower bound for the Milnor number of function germ invariant with respect to a finite abelian group action. It is shown that this bound is tight for functions of arbitrarily many variables. We also prove the function germs that reach this lower bound are equivariantly stable, i.e. invariant analogues of Morse singularities.
\end{abstract}

\section{Introduction}

In \cite{1}, M. E. Kazarian has posed a following conjecture\footnote{This is essentially a local version of the Matsumura-Monsky theorem(\cite{5}). Actually the proof by Matsumura and Monsky can be easily adapted to this statement. However, Kazarian and others did not know that, and this is a good thing, as otherwise we would not have this interesting problem on Milnor numbers to wrestle with.}:

\begin{theorem}
Consider a germ of function $f:(\mathbb{C}^n,0) \longrightarrow (\mathbb{C},0)$ with an isolated singularity at the origin and a compact Lie group $G$ of analytic diffeomorphisms $(\mathbb{C}^n,0) \longrightarrow (\mathbb{C}^n,0)$ such that $f(g(x)) = f(x) \; \forall \;g \in G \; \forall \;x \in \mathbb{C}^n$. If the Taylor series of $f$ contains no monomials of degree $\leq 2$ then $G$ is discrete.
\end{theorem}

V.A. Vassiliev has given an elementary proof of this statement based on the geometry of the Newton polytope of a germ with an infinite group of symmetries in \cite{2}. Some time later, S.P. Chulkov has given a proof (\cite{3}) based on the following inequality for the Milnor number of $f$:

\begin{theorem}
Consider a germ of function $f:(\mathbb{C}^n,0) \longrightarrow (\mathbb{C},0)$ with an isolated singularity at the origin such that the Taylor series of $f$ contains no monomials of degree $\leq 2$. If f is invariant with respect to a nontrivial action of a group of prime order $p$ then $\mu(f) \geq p-1$.
\end{theorem}

Kazarian's conjecture follows from this easily, as an infinite compact Lie group contains cyclic subgroups of any order. 

A slightly stronger inequality ($\mu(f) \geq p+1$) can be easily proven in case of a real group action using invariant morsifications as described in \cite{4}. These two inequalities are essentially the only known bounds on the Milnor number in terms of properties of the group action, so it is interesting to obtain a bound for the Milnor number in case of groups of composite order. Chulkov's original proof does not generalize (as far as we know) to groups other that $\mathbb{Z}_p$.

The main theorem of this paper is the generalization of Chulkov's inequality to arbitrary abelian groups. For a group action $\tau$ on $(\mathbb{C}^n,0)$ we will denote the length of shortest orbit of a non-fixed point by $l(\tau)$.
\begin{theorem}
Let $f$ be a germ of analytic function $(\mathbb{C}^n,0) \longrightarrow (\mathbb{C},0)$ with an isolated singularity at the origin such that the Taylor series of $f$ contains no monomials of degree $\leq 2$. Let $\tau$ be a nontrivial action of a finite abelian group $G$ on $(\mathbb{C}^n,0)$ by analytic diffeomorphisms. If $f$ is invariant with respect to $\tau$ then $\mu(f) \geq l(\tau)-1$.
\end{theorem}
In the case $G=\mathbb{Z}_p$ this reduces to the original Chulkov's inequality.

\section{Definitions and auxiliary statements}

In a suitable local system of coordinates any finite group action on $(\mathbb{C}^n,0)$ can be linearized (see \cite{6}). So we wil consider all group actions to be induced by a linear representation of $G$.
Furthermore, any linear representation of a finite abelian group is diagonalazable, hence further we will consider only actions of the form $g \circ (x_1,\ldots ,x_n) \longrightarrow (\chi_1(g) x_1,\ldots , \chi_n(g) x_n)$ with $\chi_i$ being group characters of $f$.

A function $f: (\mathbb{C}^n,0) \longrightarrow (\mathbb{C},0)$ is called \textbf{equivariant} with respect to group action $\tau$ on $(\mathbb{C}^n,0)$ if there is a character $\lambda$ of $G$ such that $f(\tau(g) \circ x)=\lambda(g)  f(x) \; \forall \; x  \; \forall \; g \; \in \; G$. $\lambda$ is called the \textbf{weight} of $f$. Invariant functions are equivariant with weight 1, of course. 

For a function germ $f$ $J_f$ will denote the \textbf{Jacobian ideal} $<\frac{\partial f}{\partial x_1}, \ldots, \frac{\partial f}{\partial x_n}>$.
$Q_f$ is the Jacobian algebra $\mathbb{C}\{x_1, \ldots, x_n\} / J_f$. $Q_f$ always admits a monomial basis, i.e. the basis consisting of monomials

If $f$ is invariant with respect to an action of a group $G$, the Jacobian ideal is invariant and therefore $G$ also acts on $Q_f$ by changes of coordinates: $g \circ h(x) = h ( g^{-1} (x))$. The representation of $G$ on $Q_f$ is called \textbf{the equvariant Milnor number} of $f$ and denoted $\mu_G(f)$. $\nu(f)$ is the multiplicity of the trivial representation in $\mu_G(f)$, i.e. the dimension of the space of invariant elements of $Q_f$. This subspace of invariant elements is denoted $Q^G_f$.

A group action on $(\mathbb{C}^n, 0)$  is called \textbf{real} if it admits an invariant quadratic form of rank $n$.

The next two theorems are taken from \cite{4}.

\begin{theorem}
If $f$ is invariant with respect to a real action of a finite group $G$ then there is a $G$-invariant deformation $f_{\lambda}$ of $f$ with only nondegenerate critical points.
\end{theorem}

The set $C$ of critical points of any invariant deformation $f_{\lambda}$ is $G$-invariant. The action of $G$ on $(\mathbb{C}^n, 0)$ induces a representation  $\theta$ of $G$ on the space of $\mathbb{C}$-valued functions on $C$ ($\theta(g) \circ h(x) = h ( g^{-1} (x))$.

\begin{theorem}If the deformation $f_{\lambda}$ has only nondegenerate critical points then the representations $\theta$ and  $\mu_G(f)$ are isomorphic.
\end{theorem}

The last statement mean in particular that $\nu(f)$ is equal to the number of orbits of critical points in a invariant Morse approximation of $f$. 
\begin{theorem}
If the main theorem is true for group action with no fixed points outside of origin, it is true for all group actions.
\end{theorem}
\begin{proof}
Consider the action $ \tau: g \circ (x_1,\ldots ,x_n) \longrightarrow \\ 
\longrightarrow ( x_1,\ldots , x_k, \chi_{k+1}(g) x_{k+1}, \ldots, \chi_n (g) x_n), \chi_{k+1}, \ldots, \chi_n \neq 1$ and an invariant function $f$ with a zero $2-jet$ and an isolated critical point at the origin. 

 The function $\tilde{f} = f + x_1^2 + \ldots + x_k^2$ can be reduced to the form $ x_1^2 + \ldots + x_k^2 + f|_{\{x_1=\ldots=x_k=0\}}$ by an application of the (equivariant version of) splitting lemma. Then $\mu(f) \geq \mu( \tilde{f} )$ and $\mu( \tilde{f} ) = \mu(  f|_{\{x_1=\ldots=x_k=0\}})$. $G$ acts on $\{x_1=\ldots=x_k=0\}$ with no fixed points outside of origin, so if the inequality holds in this case, it holds in general.
\end{proof}

Since $\nu(f)$ is equal to the number of orbits of critical points of an invariant Morse approximation, in the case of a real representation $\tau$ with no fixed points outside of origin we obtain the inequality $\mu(f) \geq (\nu(f) -1) l(\tau) + 1$. We will refer to this formula as \textbf{Roberts' inequality}.

\section{Proof of main theorem}

For a group action $ \tau: g \circ (x_1,\ldots ,x_n) \longrightarrow (\chi_1 (g) x_1,\ldots , \chi_n (g) x_n)$ we will denote $\tau_{\mathbb{R}}$ the action on $\mathbb{C}^{2n}$ defined as  $g \circ (x_1,\ldots ,x_n, y_1, \ldots, y_n) \longrightarrow \\ (\chi_1 (g) x_1,\ldots , \chi_n (g) x_n, \bar{\chi}_1 (g) y_1,\ldots , \bar{\chi}_n (g) y_n)$. 

$\tau_{\mathbb{R}}$ is always a real action with the invariant quadratic form $\Sigma_i x_i y_i$. If $\tau$ has no fixed points outside of origin then neither does $\tau_{\mathbb{R}}$.

For a function $f:(\mathbb{C}^n,0) \longrightarrow (\mathbb{C},0)$ $f \oplus f$ is the function $f(x_1, \ldots, x_n) + f(y_1, \ldots, y_n)$.

 It is easy to see that if $f$ has an isolated singularity at the origin then so does $f \oplus f$, and $\mu(f \oplus f) = (\mu(f))^2$. Furthermore, if $f$ is invariant with respect to an action $\tau$ then $f \oplus f$ is invariant with respect to the action $\tau_{\mathbb{R}}$.

\begin{theorem} If $p_1, \ldots, p_{\mu}$ form a monomial basis for the Jacobian algebra of $f$ then $p_i(x) p_j (y), 1 \leq i,j \leq \mu$ form a monomial basis in the Jacobian algebra of  $f \oplus f$.
\end{theorem}
\begin{proof} Start with an arbitrary monomial basis ${q_1(x,y), \ldots , q_{\mu^2}(x,y)}$ for $Q_{f \oplus f}$. Each $q_i$ factors as as a product of monomials in $x$ and $y$: $q_i = q_{i1}(x) q_{i2}(y)$. For each $i$ $q_{i1} \notin J_f$, as otherwise $q_i \in J_{f \oplus f}$. Hence we have a decomposition $q_{i1} = \Sigma_{j}c_{i j} p_j (x)$. Similarly, each $q_{i2}$ is a linear combination of $p_j(y)$. Therefore $p_i(x) p_j (y)$ form a system of generators for $Q_{f \oplus f}$. As there are exactly $\mu^2$ of them, they form a basis.

\end{proof}

\begin{theorem}$\nu(f \oplus f) = k_1^2 + \ldots + k_m^2$, with $k_i$ being the multiplicities of different one-dimensional representations in the irreducible decomposition of $Q_{f }$.
\end{theorem}
\begin{proof} As proven in the previous theorem $Q_{f \oplus f}$ is isomorphic ( as a vector space) to the tensor product of  $Q_f$ with itself. On the second factor of this tensor product (i.e. on monomials in $y_i$) $G$ acts by conjugate characters, so the action on the group on the second factor can be identified with the action on the conjugate space $Q^{*}_f$. According to a well-known result in representation theory (see, e.g., \cite{7}, pp. 433--434) for each finite-dimensional representation $V$ the representation $V \otimes V^{*}$is naturally isomorphic to  $Hom(V, V)$, and the invariant subspace of  $V \otimes V^{*}$  is isomorphic to the space of equivariant endomorphisms of $V$. The dimension of the space of equivariant endomorphisms is $k^2_1 + \ldots + k^2_m$ with  $k_i$ being the multiplicities of irreducible representations in the decomposition of $V$.

\end{proof}
 
 Since $\mu(f) = k_1 + \ldots + k_m$, we have, in particular, $\nu(f \oplus f) \geq \mu(f)$.
Now we can conclude the proof of the main theorem. As $\tau_{\mathbb{R}}$ is a real representation, we can use Roberts' inequality for the Milnor number of $f \oplus f$: 

$\mu(f \oplus f) \geq (\nu (f \oplus f) - 1) l( \tau _{\mathbb{R}}) +1$. 

As we now know $\mu(f \oplus f) = \mu^2 (f)$, $\nu(f \oplus f) \geq \mu(f)$ and, obviously, $l( \tau _{\mathbb{R}}) = l(\tau)$. Substituting all of these into Roberts' inequality gives us the inequality 

$\mu^2(f) \geq (\mu (f) - 1) l(\tau) +1$. 

Solving this quadratic inequality for $\mu(f)$, we obtain $\mu(f) =1$(which is impossible given that the 2-jet of $f$ is zero) or $\mu(f) \geq l(\tau) -1$. 

\section{Tightness of the bound}

The next step is to produce examples showing the tightness of this bound. The trivial example is the 1-dimensional representation: if a character $\chi$ of group $G$ has order $k$ and $G$ acts on $\mathbb{C}$ via this character, the invariant function $x^k$ has the lowest possible Milnor number $k-1$. However, this is unsatisfactory, as we would ideally like to prove that the bound is tight in arbitrarily large dimensions. To show this, we present the following examples.

Let $n$ be an odd positive integer, $d_1, \ldots, d_n \geq 2$ -- natural numbers, $g$ -- generator of a group  $\mathbb{Z}_{d_1\ldots d_n + 1}$,  $\epsilon$ -- pimitive root of unity $1$ of degree $d_1\ldots d_n + 1$.
 $\mathbb{Z}_{d_1\ldots d_n + 1}$ acts on $\mathbb{C}^n$ in the following fashion: 
 
 \begin{equation}{\nonumber}
g \circ (x_1, \ldots, x_n) = (\epsilon x_1, \epsilon^{-d_1} x_2, \ldots, \epsilon^{(-1)^{n-2} d_1 \ldots d_{n-2}} x_{n-1}, \epsilon^{(-1)^{n-1} d_1 \ldots d_{n-1}} x_n)
 \end{equation}

 The polynomial $x_1^{d_1} x_2 + x_2^{d_2} x_3 + \ldots + x_{n-1}^{d_{n_1}} x_n + x_n^{d_n} x_1$ is invariant with respect to this action. It has an isolated critical point, as it is an invertible polynomial of the loop type (see., e.g., \cite{8,9}), and its Milnor number is equal to $d_1\ldots d_n$.

The length of any nontrivial orbit in this case is $d_1\ldots d_n  +  1$, so the bound  $\mu \geq l(\tau)-1$ is tight.

\section{Equivariantly stable germs}

There are many papers on the classification of equivariantly simple singularities (see., e.g., \cite{10}, \cite{11}, \cite{12}). Roughly speaking, an \textbf{equivariantly stable} singularity is the simplest of the equivariantly simple singularities, i.e. the one that is only adjacent to itself. If the case of real group action the equivariantly stable singularity is the nondegenerate one, as the deformation of a Morse function is again a Morse function. For the formal definition of equivariant stability see below.

We will denote $\mathfrak{m}_G$ the maximal ideal in the ring of $G$-invariant analytic germs. Since our action has no fixed points outside of origin, every element of $\mathfrak{m}_G$ has a critical point at the origin. Consider $\mathcal{O}^{r}_{G}$ --- the space of $r$-jets of elements in $\mathfrak{m}_G$.  The group $D^r_G$ of $r$-jets of  $G$-equivariant diffeomorphisms  $(\mathbb{C}^n,0) \longrightarrow (\mathbb{C}^n,0)$ acts on this space.

A germ $f:(\mathbb{C}^n,0) \longrightarrow (\mathbb{C},0)$ is \textbf{equivariantly stable} if for all large enough $r$  $D^r_G$-orbit of $r$-jet of $f$ in  $\mathcal{O}^{r}_{G}$ is open.

\begin{theorem}
$f$ is equivariantly stable iff $\nu(f)=1$.

\end{theorem}
\begin{proof}
Let $f$ be equivariantly stable. As the orbit of $r$-jet of $f$ is open for large enough $r$ the tangent space to the orbit has dimension equal to the dimension of $\mathcal{O}^{r}_{G}$ itself. The tangent space to the orbit  at $r$-jet of $f$ is the projection of the jacobian ideal of $f$ to  $\mathcal{O}^{r}_{G}$. But if for all $r$ large enough the projection of $J_f$ covers all  $\mathcal{O}^{r}_{G}$ $J_f$ contains $\mathfrak{m}_G$. Since $f$ has a critical point at the origin  $J_f$ does not contain 1 and $\nu(f)=1$.

Now assume $\nu(f)=1$. The Milnor number of $f$ is finite, as the invariants of a finite group action have the separating property, i.e. for any two orbits of the $G$-action there is an invariant function that is equal to 1 on the first orbit and 0 on the second. But if $\nu(f)=1$, then $J_f = \mathfrak{m}_G$ and any invariant function takes the same value at any point of the singular locus of $f$ as it does at the origin. So the separating property fails if the critical point of $f$ is non-isolated. Since  $f$ has a finite Milnor number we can apply the Slodowy theorem (see \cite{13}), according to which the invariant versal deformation of $f$ is equal to $f + \lambda_1 h_1 + \ldots + \lambda_{\nu} h_{\nu}$, $h_j$ --- linearly independent invariant elements of $Q_f$. Since $\nu(f)=1$, the versal deformation of $f$ is trivial, i.e. has the form $f_{\lambda} = f + \lambda$, and the equivariant stability follows immediately.
\end{proof}
In particular for the nondegenerate singularity $\mu(f)=\nu(f)=1$, as $Q_f$ is equal to $\mathbb{C}$.

\begin{theorem}

If $\mu(f) = l(\tau)-1$ $f$ is equivariantly stable
\end{theorem}

\begin{proof}
If $\mu(f) = l(\tau)-1$, the inequality $\mu^2(f) \geq (\mu (f) - 1) l(\tau) +1$ becomes an equality. Since also  $\mu^2(f) \geq (\nu (f \oplus f) - 1) l(\tau) +1$ and $\nu(f \oplus f) \geq \mu(f)$, it can only become an equality if $\nu(f \oplus f) = \mu(f)$. As was proven earlier, $\nu(f \oplus f) = k_1^2 + \ldots + k_m^2$, with $k_i$ being the multiplicities of one-dimensional representations in $Q_{f }$. With $\mu(f)$ equal to $k_1 + \ldots + k_m$, the equality $\nu(f \oplus f) = \mu(f)$ can only hold if $\forall \; i \; k_i=1$. In particular, the multiplicity of the trivial representation (equal to $\nu(f)$) is 1.
\end{proof}


\begin{thebibliography}{99}

\bibitem{1} M.~E.~Kazarian, ``Characteristic classes of Lagrangian and Legendre singularities'', Russian Math. Surveys, 50:4 (1995), 701--726.
\bibitem{2}V.~A.~Vassiliev, ``On a problem by M. Kazarian'', Funct. Anal. Appl., 33:3 (1999), 220--221.
\bibitem{3}S.~P.~Chulkov, ``On the Milnor Number of an Equivariant Singularity'', Math. Notes, 71:6 (2002), 871--874.
\bibitem{4}M.~Roberts, ``Equivariant Milnor Numbers and Invariant Morse Approximations'', J.L.M.S., 31 (1985), 487--500.
\bibitem{5}H.~Matsumura, P.~Monsky, ``On the automorphisms of hypersurfaces'', J. Math. Kyoto Univ. 3:3 (1963), 347-361.
\bibitem{6} S.~Bochner, ``Compact groups of differentiable transformations'', Ann. of Math., 46:3 (1945), 372--381.

\bibitem{7} C.~B.~Hall, ``Lie Groups, Lie Algebras, and Representations: An Elementary Introduction'', Springer, 2015.

\bibitem{8} W.~Ebeling, S.~M.~Gusein-Zade, ``Saito duality between Burnside rings for invertible polynomials'', Bull. London Math. Soc., 44 (2012), 814--822.
\bibitem{9} M.~Kreuzer, H.~Skarke, ``On the classification of quasihomogeneous functions'', Commun.Math. Phys., 150 (1992), 137--147.

\bibitem{10} E.~A.~Astashov, "Classification of $\mathbb{Z}_3$-Equivariant Simple Function Germs'', Math. Notes, 105:2 (2019), 161--172.
\bibitem{11}S.~M.~Gusein-Zade, A.-M.~Ya.~Rauch, ``On simple $\mathbb{Z}_3$
-invariant function germs'', Funct. Anal. Appl., 55:1 (2021), 45--51.



\bibitem{12} I.~A.~Proskurnin, ``Singularities Equivariantly Simple with Respect to Irreducible Representations'', Funct. Anal. Appl., 57:1 (2023), 60--64.
\bibitem{13} P.~Slodowy, ``Einige Bemerkungen zur Entfaltung symmetrischer Funktionen'', Math.~Z., 158:2 (1978), 157--170.

\end{thebibliography}
\end{document}